\documentclass[reqno]{amsart}
\usepackage{amsmath,amscd,amsthm,amsfonts,amsopn,amssymb,mathrsfs}

\newtheorem{theorem}{Theorem}

\newtheorem{lemma}[theorem]{Lemma}
\theoremstyle{remark}

\allowdisplaybreaks
\begin{document}

\title{Global well-posedness for supercritical SQG with perturbations of radially symmetric data}

\author{Aynur Bulut}
\address{Department of Mathematics, Louisiana State University, 303 Lockett Hall, Baton Rouge, LA 70803}
\email{aynurbulut@lsu.edu}

\author{Hongjie Dong}
\address{Division of Applied Mathematics, Brown University, Providence, RI, 02912.}
\email{Hongjie\_Dong@brown.edu}

\thanks{\today}

\begin{abstract}
We study the global well-posedness of the supercritical dissipative surface quasi-geostrophic (SQG) equation, a key model in geophysical fluid dynamics. While local well-posedness is known, achieving global well-posedness for large initial data remains open. Motivated by enhanced decay in radial solutions, we aim to establish global well-posedness for small perturbations of potentially large radial data. Our main result shows that for small perturbations of radial data, the SQG equation admits a unique global solution.
\end{abstract}

\maketitle

\section{Introduction}

Fix $0<\gamma<1$ and set $s=s(\gamma):=2-\gamma$.  In this short note, we consider the supercritical dissipative surface quasi-geostrophic (SQG) equation,
\begin{align}
\partial_t\theta+R^\perp\theta\cdot\nabla\theta+\Lambda^\gamma\theta=0,\label{eq-sqg}
\end{align}
posed on $\mathbb{R}^2$, where $\Lambda:=\sqrt{-\Delta}$, and $R^\perp$ is defined by $$R^\perp:\theta\mapsto (R_2\theta,-R_1\theta),$$
with $R_i:=\partial_i\Lambda^{-1}$, $i\in \{1,2\}$ denoting the $i$th component of the vector-valued Riesz transform on $\mathbb{R}^2$.

The SQG equation ($\ref{eq-sqg}$) arises as a fundamental mathematical model in the study of rapidly rotating geophysical fluid dynamics.  From a mathematical perspective, it exhibits a qualitative similarity to the Euler and Navier-Stokes systems. While the local well-posedness of the supercritical SQG in the critical space $H^s$ and the small data global well-posedness were established long ago by Miura \cite{M06} and Ju \cite{NJ22}, the global-in-time well-posedness of the equation with arbitrary large initial data remains a challenge. We refer readers to \cite{CW,DP09} for conditional regularity results, \cite{Sil,Dab} for eventual regularity of weak solutions, and \cite{DKV,XZ,DKSV,CV16} for global well-posedness for equations with slightly supercritical dissipations (see also \cite{BD21} and \cite{BHP23} and the references cited therein for recent results on the forced SQG in a variety of settings).

We observe that when the initial data is radially symmetric, the solution remains radial, rendering the nonlinear term in \eqref{eq-sqg} vanishing. Consequently, the solution satisfies a linear fractional heat equation, exhibiting global existence and explicit expression using the fractional heat kernel. Moreover, such a solution instantaneously smoothens and decays in the time variable.

In this paper, our goal is to examine the extent to which the improved decay enjoyed by radial solutions allows one to obtain a global well-postedness result for data which is a {\it small perturbation} (in a suitably chosen norm) of the class of radial data.  In particular, we establish the following result.

\begin{theorem}
\label{thm1}
Fix $0<\gamma<1$ and set $s=2-\gamma$.  Then there exists $\epsilon>0$ such that if $f\in L^1(\mathbb{R}^2)\cap H^s(\mathbb{R}^2)$ is radially symmetric with $f\geq 0$, and $g\in H^{2}(\mathbb{R}^2)$ satisfies $$\lVert g\rVert_{H^{2}}\leq \epsilon$$ then the initial value problem
\begin{align}
\left\lbrace\begin{array}{c}\partial_t\theta+u\cdot\nabla\theta+\Lambda^\gamma \theta=0\\u=R^\perp\theta\\\theta(0,x)=f(x)+g(x)\end{array}\right.\label{eq-gwp-1}
\end{align}
has a unique global solution in $C([0,\infty);H^{2-\gamma}(\mathbb{R}^2))$.
\end{theorem}

The proof of Theorem \ref{thm1} relies on a decomposition argument, splitting the solution $\theta$ into a radial part $\theta_0$, solving the fractional heat equation with initial data $f$, and a perturbation part $\theta_1$, shown to be small for all time. For the radial part, we utilize smoothing and decay estimates, while for the perturbation part, we employ energy-type arguments akin to \cite{NJ22}. Finally, we leverage a continuity argument to close the estimates.

Because of the importance of the dissipation effects in our arguments, we end this introduction by recalling some decay estimates for the fractional heat equation on $\mathbb{R}^2$.  In particular, for any $r\in [1,2]$ and $s\geq 0$, we have
\begin{align}
\lVert e^{-|\nabla|^\gamma t}f\rVert_{L^2}&\lesssim t^{-\frac{2}{\gamma}(\frac{1}{r}-\frac{1}{2})}\lVert f\rVert_{L^r},\label{eq-heat1}
\end{align}
and
\begin{align}
\lVert e^{-|\nabla|^\gamma t}f\rVert_{\dot{H}^s}&\lesssim t^{-\frac{s}{\gamma}}\lVert f\rVert_{L^2}\label{eq-heat2}.
\end{align}
The estimates ($\ref{eq-heat1}$) and ($\ref{eq-heat2}$) can be established by appealing to the convolution representation of solutions and the decay properties of the heat kernel on $\mathbb{R}^2$, combined with Young's inequality; see, for instance, \cite[Lemma 3.1]{MYZ}.

We conclude this introduction by briefly recalling a few papers from the literature which treat decay and stability results that are of note in the context of the statement and proof of Theorem $\ref{thm1}$.  In the log-supercritical SQG regime, \cite{C23} establishes a class of decay results for H\"older norms, while in the fully supercritical setting, \cite{BK21} establishes decay in the scaling critical norm for solutions evolving from small initial data.  In the context of initial data of larger norm, in \cite{LPW19} the authors establish a supercritical SQG global well-posedness result for perturbations of initial data having Fourier support away from the origin; in this respect, the result of \cite{LPW19} is conceptually similar to the radial perturbation result presented in Theorem $\ref{thm1}$, though the norms involved in \cite{LPW19} are not scaling critical.

\subsection*{Acknowledgements}

H. Dong was partially supported by the NSF under agreement DMS-2055244.

\section{Motivating decomposition: perturbation lemma}
\label{sec2}

Let $f\in H^s(\mathbb{R}^2)$ and $g\in H^2(\mathbb{R}^2)$ be given such that $f$ is radially symmetric.  Our goal is to show that if $g$ is sufficiently small in $H^2$ norm, then solutions to the initial value problem (\ref{eq-gwp-1}) can be extended globally in time.  For this, we will use a class of a priori estimates motivated by a perturbation argument.

In particular, letting $K_\gamma(t)$ be the fractional heat kernel defined by $\widehat{K_\gamma(t)}(\xi)=\exp(-|\xi|^{\gamma}t)$ for $t>0$ and $\xi\in\mathbb{R}^2$, we will view solutions to \eqref{eq-gwp-1} as perturbations of $$\theta_0:t\mapsto K_\gamma(t)\ast f,\quad t>0,$$
that is, the solution to
\begin{align*}
\left\lbrace\begin{array}{c}\partial_t\theta_0+\Lambda^\gamma \theta_0=0\\
\theta_0(0)=f.\end{array}\right.
\end{align*}

Indeed, letting $\theta$ solve \eqref{eq-gwp-1} and writing $$\theta=\theta_1+\theta_0,$$ the perturbation $\theta_1$ solves
\begin{align}
\partial_t\theta_1+\Lambda^\gamma\theta_1&=-(R^\perp\theta_0)\cdot\nabla\theta_1-(R^\perp\theta_1)\cdot\nabla\theta_1-(R^\perp\theta_1)\cdot\nabla\theta_0,\label{eq-theta1}
\end{align}
with $\theta_1(0)=g$.

We now introduce a lemma giving a priori bounds on solutions to the perturbed equation \eqref{eq-theta1}.

\begin{lemma}
\label{lem1}
Let $f$, $\theta_0$ be as above, and suppose that $\theta_1$ solves \eqref{eq-theta1}.  Then, for all $t>0$, we have
\begin{align*}
&\frac{d}{dt}\Big[\frac{1}{2}\lVert \Lambda^s\theta_1\rVert_{L^2}^2\Big]+\frac{1}{4}\lVert \Lambda^{s+\frac{\gamma}{2}}\theta_1\rVert_{L^2}^2\\
&\hspace{0.2in}\lesssim \lVert \Lambda^s\theta_1\rVert_{L^2}^2\lVert \Lambda^{s+\frac{\gamma}{2}}\theta_0\rVert_{L^2}^2\\
&\hspace{0.4in} +(\lVert\theta_1\rVert_{L^2}^2+\lVert \Lambda^s\theta_1\rVert_{L^2}^2)\lVert \Lambda^{s+\frac{\gamma}{2}+\sigma}\theta_0\rVert_{L^2}^2\\
&\hspace{0.4in}+\lVert \Lambda^s\theta_1\rVert_{L^2}\lVert \Lambda^{s+\frac{\gamma}{2}}\theta_1\rVert_{L^2}^2.
\end{align*}
\end{lemma}

\begin{proof}
Letting $\theta_1$ be as in the statement of the lemma, we observe that for all $t>0$ we have
\begin{align}
\nonumber \frac{d}{dt}\bigg[\frac{1}{2}\lVert \Lambda^s\theta_1\rVert_{L^2}^2\bigg]&=-\int (\Lambda^s\theta_1)\Lambda^s[\Lambda^{\gamma}\theta_1+(R^\perp\theta_0)\cdot\nabla\theta_1+(R^\perp\theta_1)\cdot\nabla\theta_0]dx\\
\nonumber &\leq -\int \Lambda^s\theta_1\Lambda^{s+\gamma}\theta_1dx\\
\nonumber &\hspace{0.2in}+\bigg|\int (\Lambda^s\theta_1)\Big(\Lambda^s[(R^\perp\theta_0)\cdot\nabla\theta_1]-(R^\perp\theta_0)\cdot\nabla\Lambda^s\theta_1\Big)dx\bigg|\\
\nonumber &\hspace{0.2in}+\bigg|\int (\Lambda^s\theta_1)\Big(\Lambda^s[(R^\perp\theta_1)\cdot\nabla\theta_1]-(R^\perp\theta_1)\cdot\nabla\Lambda^s\theta_1\Big)dx\bigg|\\
&\hspace{0.2in}+\bigg|\int (\Lambda^s\theta_1)\Lambda^s[(R^\perp\theta_1)\cdot\nabla\theta_0]dx\bigg|,\label{eq-A}
\end{align}
where we have used the identities $$\int (\Lambda^s\theta_1)((R^\perp\theta_i)\cdot \nabla\Lambda^s\theta_1)dx=0,\quad i\in\{0,1\},$$ which follow from integration by parts.  Noting that integration by parts also gives $$\langle \Lambda^s\theta_1,\Lambda^{s+\gamma}\theta_1\rangle=\lVert \Lambda^{s+\gamma/2}\theta_1\rVert_{L^2}^2,$$ and $$\langle \Lambda^s\theta_1,\Lambda^s[(R^\perp\theta_1)\cdot\nabla\theta_0]\rangle=\langle \Lambda^{s-\frac{\gamma}{2}}\theta_1,\Lambda^{s+\frac{\gamma}{2}}[(R^\perp\theta_1)\cdot\nabla\theta_0\rangle,$$ we obtain the bound
\begin{align}
\nonumber (\ref{eq-A})&\leq -\lVert \Lambda^{s+\frac{\gamma}{2}}\theta_1\rVert_{L^2}^2+\lVert \Lambda^s\theta_1\rVert_{L^2}\lVert \Lambda^s[(R^\perp\theta_0)\cdot\nabla\theta_1]-(R^\perp\theta_0)\cdot\nabla\Lambda^s\theta_1\rVert_{L^2}\\
\nonumber &\hspace{0.2in}+\lVert \Lambda^s\theta_1\rVert_{L^2}\lVert \Lambda^s[(R^\perp\theta_1)\cdot\nabla\theta_1]-(R^\perp\theta_1)\cdot\nabla\Lambda^s\theta_1\rVert_{L^2}\\
\nonumber &\hspace{0.2in}+\lVert \Lambda^{s+\frac{\gamma}{2}}\theta_1\rVert_{L^2}\lVert \Lambda^{s-\frac{\gamma}{2}}[(R^\perp\theta_1)\cdot\nabla\theta_0]\rVert_{L^2}\\
\nonumber &\leq -\frac{1}{2}\lVert \Lambda^{s+\frac{\gamma}{2}}\theta_1\rVert_{L^2}^2+\lVert \Lambda^s\theta_1\rVert_{L^2}\lVert \Lambda^s[(R^\perp\theta_0)\cdot\nabla\theta_1]-(R^\perp\theta_0)\cdot\nabla\Lambda^s\theta_1\rVert_{L^2}\\
\nonumber &\hspace{0.2in}+\lVert \Lambda^s\theta_1\rVert_{L^2}\lVert \Lambda^s[(R^\perp\theta_1)\cdot\nabla\theta_1]-(R^\perp\theta_1)\cdot\nabla\Lambda^s\theta_1\rVert_{L^2}\\
&\hspace{0.2in}+C\lVert \Lambda^{s-\frac{\gamma}{2}}[(R^\perp\theta_1)\cdot\nabla\theta_0]\rVert_{L^2}^2.\label{eq-B}
\end{align}

We now invoke the procedure of commutator estimates as in \cite{NJ22}.  In particular, recall that we have
\begin{align*}
\nonumber&\lVert \Lambda^s[h\cdot k]-h\cdot\Lambda^s k\rVert_{L^2}\\
&\hspace{0.2in}\lesssim \lVert \nabla h\rVert_{L^{4/\gamma}}\lVert \Lambda^{s-1} k\rVert_{L^{4/(2-\gamma)}} +\lVert \Lambda^s h\rVert_{L^{4/(2-\gamma)}}\lVert k\rVert_{L^{4/\gamma}}
\end{align*}
for all $h$ and $k$ for which the right-hand side is finite.  Applying this with $h:=R^\perp\theta_0$ and $k:=\nabla\theta_1$, we obtain that the right-hand side of \eqref{eq-B} is bounded by
\begin{align}
\nonumber & -\frac{1}{2}\lVert \Lambda^{s+\frac{\gamma}{2}}\theta_1\rVert_{L^2}^2+C\cdot\bigg(\lVert \Lambda^s\theta_1\rVert_{L^2}\lVert \Lambda \theta_0\rVert_{L^{4/\gamma}}\lVert \Lambda^{s}\theta_1\rVert_{L^{4/(2-\gamma)}}\\
\nonumber &\hspace{2.0in}+\lVert \Lambda^s\theta_1\rVert_{L^2}\lVert \Lambda^s\theta_0\rVert_{L^{4/(2-\gamma)}}\lVert \Lambda \theta_1\rVert_{L^{4/\gamma}}\\
\nonumber &\hspace{2.0in}+\lVert \Lambda^s\theta_1\rVert_{L^2}\lVert \Lambda \theta_1\rVert_{L^{4/\gamma}}\lVert \Lambda^{s}\theta_1\rVert_{L^{4/(2-\gamma)}}\\
\nonumber &\hspace{2.0in}+\lVert \Lambda^s\theta_1\rVert_{L^2}\lVert \Lambda^s\theta_1\rVert_{L^{4/(2-\gamma)}}\lVert \Lambda \theta_1\rVert_{L^{4/\gamma}}\\
&\hspace{2.0in}+\lVert \Lambda^{s-\frac{\gamma}{2}}[(R^\perp\theta_1)\cdot\nabla\theta_0]\rVert_{L^2}^2\bigg).\label{eq-B1}
\end{align}
The Sobolev embedding then gives
\begin{align*}
(\ref{eq-B1})&\leq -\frac{1}{2}\lVert \Lambda^{s+\frac{\gamma}{2}}\theta_1\rVert_{L^2}^2+C\cdot \bigg(\lVert \Lambda^s\theta_1\rVert_{L^2}\lVert \Lambda^{s+\frac{\gamma}{2}}\theta_0\rVert_{L^2}\lVert \Lambda^{s+\frac{\gamma}{2}}\theta_1\rVert_{L^2}\\
&\hspace{2.2in}+\lVert \Lambda^s\theta_1\rVert_{L^2}\lVert \Lambda^{s+\frac{\gamma}{2}}\theta_1\rVert_{L^2}^2\\
&\hspace{2.2in}+\lVert \Lambda^{s-\frac{\gamma}{2}}[(R^\perp\theta_1)\cdot\nabla\theta_0]\rVert_{L^2}^2\bigg)\\
&\leq -\frac{1}{4}\lVert \Lambda^{s+\frac{\gamma}{2}}\theta_1\rVert_{L^2}^2+C\cdot\bigg(\lVert \Lambda^s\theta_1\rVert_{L^2}^2\lVert \Lambda^{s+\frac{\gamma}{2}}\theta_0\rVert_{L^2}^2+\lVert \Lambda^s\theta_1\rVert_{L^2}\lVert \Lambda^{s+\frac{\gamma}{2}}\theta_1\rVert_{L^2}^2\\
&\hspace{2.2in}+\lVert \Lambda^{s-\frac{\gamma}{2}}[(R^\perp\theta_1)\cdot\nabla\theta_0]\rVert_{L^2}^2\bigg).
\end{align*}

To estimate the second term on the right-hand side of this bound, fix $\sigma>0$ with $1-\gamma<\sigma<2-\gamma$ as a parameter to be determined later in the argument, and set $(q,r)=(\frac{2}{\sigma-(1-\gamma)},\frac{2}{2-\gamma-\sigma})$.  Then, by the fractional product rule, the boundedness of Riesz transforms, and the Sobolev embeddings, we have
\begin{align}
\nonumber &\lVert \Lambda^{s-\frac{\gamma}{2}}[(R^\perp\theta_1)\cdot \nabla\theta_0]\rVert_{L^2}\\
\nonumber &\hspace{0.2in}\lesssim \lVert \Lambda^{s-\frac{\gamma}{2}}\theta_1\rVert_{L^{4/(2-\gamma)}}\lVert \Lambda\theta_0\rVert_{L^{4/\gamma}}+\lVert \theta_1\rVert_{L^{q}}\lVert \Lambda^{s-\frac{\gamma}{2}+1}\theta_0\rVert_{L^{r}}\\
&\hspace{0.2in}\lesssim \lVert \Lambda^{s}\theta_1\rVert_{L^2}\lVert \Lambda^{s+\frac{\gamma}{2}}\theta_0\rVert_{L^2}+\lVert \Lambda^{s-\sigma}\theta_1\rVert_{L^2}\lVert \Lambda^{s+\frac{\gamma}{2}+\sigma}\theta_0\rVert_{L^2}.\label{eq-C}
\end{align}

Now, by standard interpolation inequalities,
\begin{align*}
(\ref{eq-C})&\lesssim \lVert \Lambda^{s}\theta_1\rVert_{L^2}\lVert \Lambda^{s+\frac{\gamma}{2}}\theta_0\rVert_{L^2}\\
&\hspace{0.2in}+\lVert \theta_1\rVert_{L^2}^{\sigma/s}\lVert \Lambda^s\theta_1\rVert_{L^2}^{1-(\sigma/s)}\lVert \Lambda^{s+\frac{\gamma}{2}+\sigma}\theta_0\rVert_{L^2}.
\end{align*}

Applying Young's inequality and collecting the above estimates, we obtain the desired bound.
\end{proof}

\section{Proof of Theorem \ref{thm1}}

We are now ready to give the proof of Theorem $\ref{thm1}$.  The argument is based on a combination of the a priori bounds expressed in Lemma $\ref{lem1}$ for the perturbed equation \eqref{eq-theta1} with

\begin{proof}[Proof of Theorem \ref{thm1}]
Let $s=2-\gamma$ be as in the statement.  Fix $\epsilon\in (0,1)$ to be determined later in the argument, and suppose that $f\in L^1(\mathbb{R}^2)\cap H^s(\mathbb{R}^2)$ and $g\in H^2(\mathbb{R}^2)$ are such that $f$ is radially symmetric and $\lVert g\rVert_{H^2}\leq \epsilon$.

By the local well-posedness theory for the supercritical SQG equation ($\ref{eq-gwp-1}$), there exists $T_1>0$ such that a unique smooth solution $\theta$ to the initial value problem ($\ref{eq-gwp-1}$) exists on the interval $[0,T_1]$.  Now, let $\theta_0:t\mapsto K_\gamma(t)\ast f$ be as in Section \ref{sec2} and set $$\theta_1(t):=\theta(t)-\theta_0(t)$$ for $t\in [0,T_1]$.  An application of Lemma \ref{lem1} then gives the bound
\begin{align}
\nonumber &\frac{d}{dt}\Big[\frac{1}{2}\lVert \Lambda^s\theta_1\rVert_{L^2}^2\Big]+\frac{1}{4}\lVert \Lambda^{s+\frac{\gamma}{2}}\theta_1\rVert_{L^2}^2\\
\nonumber &\hspace{0.2in}\lesssim \lVert \Lambda^s\theta_1\rVert_{L^2}^2\lVert \Lambda^{s+\frac{\gamma}{2}}\theta_0\rVert_{L^2}^2+(\lVert\theta_1\rVert_{L^2}^2+\lVert \Lambda^s\theta_1\rVert_{L^2}^2)\lVert \Lambda^{s+\frac{\gamma}{2}+\sigma}\theta_0\rVert_{L^2}^2\\
&\hspace{0.4in}+\lVert \Lambda^s\theta_1\rVert_{L^2}\lVert \Lambda^{s+\frac{\gamma}{2}}\theta_1\rVert_{L^2}^2.\label{eq-est1}
\end{align}
Moreover, multiplying \eqref{eq-theta1} by $\theta_1$ and integrating gives
\begin{align}
&\frac{d}{dt}\Big[\frac{1}{2}\lVert \theta_1\rVert_{L^2}^2\Big]+\lVert \Lambda^{\frac{\gamma}{2}}\theta_1\rVert_{L^2}^2\lesssim \lVert \theta_1\rVert_{L^2}^2\lVert \theta_0\rVert_{L^\infty}.\label{eq-est2}
\end{align}

Our goal is to find $t_2>0$ so that $\theta_0(t_2)+\theta_1(t_2)$ has small $\dot{H}^s$ norm (yielding access to the small data global well-posedness theory for $\theta$).  For this, note that by the local theory for \eqref{eq-gwp-1}, we can choose $t_1\in (0,T_1]$ independent of $\epsilon\in (0,1)$ so that $$\lVert \theta_1(t_1)\rVert_{H^s}\leq 2\epsilon.$$
Indeed, if $\theta_1(0,\cdot)\equiv 0$, then it remains to be zero at $t_1$.
Now if $\theta_1(0,\cdot)$ is sufficiently small in $H^s$, then $\theta(0,\cdot)$ is sufficiently close to $\theta_0(0,\cdot)$ in $H^s$. By the continuity dependence of the solution with respect to the initial data, $\theta(t_1,\cdot)$ is sufficiently close to $\theta_0(t_1,\cdot)$ in $H^s$, and thus the $H^s$ norm of $\theta_1(t_1,\cdot)$ can be as small as we want provided that $\theta_1(0,\cdot)$ is sufficiently small.

On the other hand, the smoothness of $K_\gamma(t)$ for all $t>0$ implies that one has $\theta_0(t_1)\in C^\infty(\mathbb{R}^2)$, so that for all $t>t_1$ one has
\begin{align*}
\lVert \theta_0(t)\rVert_{L^2}&\leq \lVert \theta_0(t_1)\rVert_{L^2},\quad \lVert \theta_0(t)\rVert_{\dot{H}^s}\leq \lVert \theta_0(t_1)\rVert_{\dot{H}^s},
\end{align*}
and
\begin{align*}
\lVert \theta_0(t)\rVert_{L^\infty}\leq \lVert \theta_0(t_1)\rVert_{L^\infty}.
\end{align*}
We now combine these estimates with the decay enjoyed by $\theta_0$ as a result of \eqref{eq-heat1} and \eqref{eq-heat2} to establish the global result, provided $\epsilon$ is chosen sufficiently small.  In particular, let $\epsilon_0>0$ be chosen so that solutions to \eqref{eq-gwp-1} corresponding to data with $H^s$ norm less than $\epsilon_0$ are global.  Now, choose $t_2>t_1$ large enough so that $$\lVert\theta_0(t_2)\rVert_{H^s}<\frac{\epsilon_0}{2}.$$

Returning to estimates for $\theta_1$, we invoke a continuity argument on the $H^s$ norm, based on the estimates \eqref{eq-est1} and \eqref{eq-est2}.  Note that we can find $C_1=C_1(t_1)>0$ such that
\begin{align*}
\frac{d}{dt}\Big[\frac{1}{2}\lVert \theta_1\rVert_{H^s}^2\Big]+\frac{1}{4}\lVert \theta_1\rVert_{\dot{H}^{\frac{\gamma}{2}}\cap \dot{H}^{s+\frac{\gamma}{2}}}^2\leq C_1\lVert \theta_1\rVert_{H^s}^2+C_1\lVert \theta_1\rVert_{\dot{H}^s}\lVert \theta_1\rVert_{\dot{H}^{s+\frac{\gamma}{2}}}^2.
\end{align*}

Now, choose a parameter $\epsilon_1<1/(16C_1)$ and set $I:=\{T\in (t_1,t_2]:\lVert \theta_1(t)\rVert_{H^s}\leq \epsilon_1\,\,\textrm{for all}\,\,t\in [t_1,T]\}$.  It follows from the local well-posedness theory for \eqref{eq-theta1} that $I$ is a nonempty closed subset of $(t_1,t_2]$.  Moreover, for all $t\in I$ with $t<t_2$ the local theory also implies that we can choose $t_*\in (t,t_2]$ such that $\sup_{t<s<t_*} \lVert \theta_1(s)\rVert_{H^s}\leq 2\epsilon_1$, and thus for all $t_1<\sigma<t_*$
\begin{align*}
\frac{d}{dt}\Big[\frac{1}{2}\lVert \theta_1\rVert_{H^s}^2\Big]+\frac{1}{8}\lVert \theta_1\rVert_{\dot{H}^{\frac{\gamma}{2}}\cap \dot{H}^{s+\frac{\gamma}{2}}}^2\leq C_1\lVert \theta_1\rVert_{H^s}^2.
\end{align*}
This in turn yields, for $t_1<\sigma<t_*$,
\begin{align*}
\lVert \theta_1(\sigma)\rVert_{H^s}\leq \exp(C_1(\sigma-t_1))\lVert \theta_1(t_1)\rVert_{H^s}\leq 2\exp(C_1(t_2-t_1))\epsilon,
\end{align*}
where we have recalled that our choice of $t_*$ gave $t_*<t_2$.

It now follows that if $\epsilon$ is chosen sufficiently small to ensure $$2\exp(C_1(t_2-t_1))\epsilon<\epsilon_1,$$ we obtain $\sigma\in I$ for all $t_1<\sigma<t_*$.  Thus $I$ is open in $(t_1,t_2]$, so that $I=(t_1,t_2]$ and $$\sup_{t_1<t<t_2} \lVert \theta_1(t)\rVert_{H^s}\leq \frac{1}{16C_1}.$$

Repeating the above arguments, we obtain $$\lVert \theta_1(t_2)\rVert_{H^s}\leq 2\exp(C_1(t_2-t_1))\epsilon,$$ so that if $\epsilon$ is additionally chosen small enough to ensure $$2\exp(C_1(t_2-t_1))\epsilon<\epsilon_0/2,$$ the small-data global theory for the SQG equation \eqref{eq-gwp-1} with data $\theta_0(t_2)+\theta_1(t_2)$ at time $t_2$ applies, and we obtain the desired global solution.
\end{proof}


\begin{thebibliography}{99}

\bibitem{BK21} J. Benameur and C. Katar.  Asymptotic study of supercritical surface Quasi-Geostrophic equation in critical space.  Nonlinear Anal. 224 (2022), 113074.

\bibitem{BD21} A. Bulut and H. Dong.  Nonlinear instability for the surface quasi-geostrophic equation in the supercritical regime.  Communications in Mathematical Physics 384 (2021), 1679--1707.

\bibitem{BHP23} A. Bulut, M.K. Huynh and S. Palasek.  Non-uniqueness up to the Onsager threshold for the forced SQG equation.  Preprint (2023), arXiv:2310.12947.

\bibitem{C23} H. Choi.  Global well-posedness of slightly supercritical SQG equations and gradient estimate.  Nonlinearity 36 (2023), no. 5, 2166--2192.

\bibitem{CW} P. Constantin and J. Wu.  Regularity of H\"older continuous solutions of the supercritical quasi-geostrophic eqaution, \textit{Ann. Inst. H. Poincar\'e Analyse Non Lin\'eaire} \textbf{25} (2008), 1103--1110.
    
\bibitem{CV16} M. Coti Zelati and V. Vicol, On the global regularity for the supercritical SQG equation,
Indiana Univ. Math. J. 65 (2016), no. 2, 535--552.

\bibitem{Dab} M. Dabkowski.  Eventual regularity of the solutions to the supercritical dissipative quasi-geostrophic equation. \textit{Geom. Funct. Anal.} 21 (2011), no. 1, 1--13.
    
\bibitem{DKV} M. Dabkowski, A. Kiselev, V. Vicol, Global well-posedness for a slightly supercritical surface quasi-geostrophic equation. Nonlinearity 25 (2012), no. 5, 1525--1535

\bibitem{DKSV} M. Dabkowski, A. Kiselev, L. Silvestre and V. Vicol.  Global well-posedness of slightly supercritical active scalar equations, \textit{Anal. PDE} \textbf{7} (2014), no. 1, 43--72.

\bibitem{DP09} H. Dong and N. Pavlović, A regularity criterion for the dissipative quasi-geostrophic equations. Ann. Inst. H. Poincaré C Anal. Non Linéaire 26 (2009), no. 5, 1607--1619.

\bibitem{NJ22} N. Ju.  Dissipative 2D Quasi-geostrophic Equation: Local Well-posedness, Global Regularity and Similarity Solutions.  Ind. Univ. Math. J 56 (2007), no. 1, 187--206.

\bibitem{LPW19} J. Liu, K. Pan and J. Wu.  A class of large solutions to the supercritical surface quasi-geostrophic equation.  Nonlinearity 32 (2019), no. 12, 5049--5059.

\bibitem{MYZ} C. Miao, B. Yuan, B. Zhang.  Well-posedness of the Cauchy problem for the fractional power dissipative equations.  Nonl. Anal. 68 (2008), 461--484.

\bibitem{M06} H. Miura, Dissipative quasi-geostrophic equation for large initial data in the critical Sobolev space, \textit{Comm. Math. Phys.} \textbf{267} (2006), no. 1, 141--157.
    
\bibitem{Sil} L. Silvestre.  Eventual regularization for the slightly supercritical quasi-geostrophic equation. \textit{Ann. Inst. H. Poincaré Anal. Non Linéaire} 27 (2010), no. 2, 693--704.
    
\bibitem{XZ} L. Xue and X. Zheng.  Note on the well-posedness of a slightly supercritical surface quasi-geostrophic equation.  \textit{J. Diff. Eq.} 253 (2012), no. 2, 795--813.
\end{thebibliography}
\end{document}